\documentclass [11pt]{amsart}

\usepackage{latexsym}
\usepackage{amsthm, amsmath, amssymb,latexsym,amsfonts,verbatim}

\usepackage{paralist}
\usepackage[mathscr]{eucal}
\usepackage[pdftex]{hyperref}
\hypersetup{citecolor=blue,linktocpage}
\usepackage[colorinlistoftodos]{todonotes}
\newtheorem{theorem}{Theorem}[section]
\newtheorem{lemma}[theorem]{Lemma}

\newtheorem{proposition}[theorem]{Proposition}
\newtheorem{remark}[theorem]{Remark}
\newtheorem{definition}[theorem]{Definition}

\DeclareMathOperator{\tr}{trace}

\newcommand{\nc}{\newcommand} 
\nc{\cH}{{\mathcal H}}
\nc{\cA}{{\mathcal A}}
\nc{\cG}{{\mathcal G}}
\nc{\cC}{{\mathcal C}}
\nc{\cO}{{\mathcal O}}
\nc{\cI}{{\mathcal I}}
\nc{\cB}{{\mathcal B}}
\nc{\cY}{{\mathcal Y}}
\nc{\cK}{{\mathcal K}} 
\nc{\cX}{{\mathcal X}}
\nc{\cS}{{\mathcal S}}
\nc{\cE}{{\mathcal E}}
\nc{\cF}{{\mathcal F}}
\nc{\cZ}{{\mathcal Z}}
\nc{\cQ}{{\mathcal Q}}
\nc{\cN}{{\mathcal N}}
\nc{\cP}{{\mathcal P}}
\nc{\cL}{{\mathcal L}}
\nc{\cM}{{\mathcal M}}
\nc{\cT}{{\mathcal T}}
\nc{\cW}{{\mathcal W}}
\nc{\cU}{{\mathcal U}}
\nc{\cJ}{{\mathcal J}}
\nc{\cV}{{\mathcal V}}
\nc{\bH}{{\mathbb H}}
\nc{\bA}{{\mathbb A}}
\nc{\bG}{{\mathbb G}}
\nc{\bC}{{\mathbb C}}
\nc{\bO}{{\mathbb O}}
\nc{\bI}{{\mathbb I}}
\nc{\bB}{{\mathbb B}}
\nc{\bY}{{\mathbb Y}}
\nc{\bK}{{\mathbb K}} 
\nc{\bX}{{\mathbb X}}
\nc{\bS}{{\mathbb S}}
\nc{\bE}{{\mathbb E}}
\nc{\bF}{{\mathbb F}}
\nc{\bZ}{{\mathbb Z}}
\nc{\bQ}{{\mathbb Q}}
\nc{\bN}{{\mathbb N}}
\nc{\bP}{{\mathbb P}}
\nc{\bL}{{\mathbb L}}
\nc{\bM}{{\mathbb M}}
\nc{\bT}{{\mathbb T}}
\nc{\bW}{{\mathbb W}}
\nc{\bU}{{\mathbb U}}
\nc{\bD}{{\mathbb D}}
\nc{\bJ}{{\mathbb J}}
\nc{\bV}{{\mathbb V}}
\nc{\bbZ}{{\mathbb Z}}
\nc{\bR}{{\mathbb R}}
\nc{\fr}{{\rightarrow}}
\nc{\co}{{\nabla}}
\newcommand{\la}{\longrightarrow}
\nc{\cu}{{\barline{\nabla}}}

\DeclareMathOperator{\codim}{codim}

\title { The   Hodge number  $h^{1,1}$  of irregular algebraic surfaces}

\author{Andrea Causin}
\address{D.A.D.U.,
Universit\`a   di Sassari, Piazza Duomo 6, 07041 Alghero(SS), Italy}

\email{acausin@uniss.it}

\author{Margarida Mendes Lopes}
\address{Departamento de Matem\'atica, Instituto Superior T\'ecnico, Universidade de Lisboa, Av. Rovisco Pais,
1049-001 Lisboa, Portugal}
\email{mmlopes@math.tecnico.ulisboa.pt }

\author{Gian Pietro Pirola}
\address{Dipartimento di Matematica,
Universit\`a di Pavia,
Via Ferrata, 1,
 27100 Pavia, Italy}
\email{gianpietro.pirola@unipv.it }

\thanks{
  This research was partially supported by  FCT (Portugal) through program POCTI/FEDER, Project 
PTDC/MAT/099275/2008, Project  PTDC/MAT-GEO/0675/201,  and by  PRIN2012 ``Moduli, stutture geometriche e loro applicazione'' and   Regione Autonoma della Sardegna l.r.~7/2007 ``Stima delle deformazioni del SRT'' (Italy).  The first author and third authors  are members of GNSAGA.-INdAM,
and the second author is a member of the Center for Mathematical
Analysis, Geometry and Dynamical Systems (IST/UL) }\smallskip

\begin{document}

       \begin{abstract}   We prove a new inequality for the Hodge number $h^{1,1}$ of irregular complex smooth projective surfaces of general type without irregular pencils of genus $\geq 2$. More specifically  we show that if the irregularity $q$ satisfies $q=2^k+1$ then $h^{1,1}\geq 4q-3$. This generalizes results previously known for $q=3$ and $q=5$.
\medskip


\end{abstract}
                   
\maketitle

\noindent

\section{Introduction}

Let $S$ be a complex smooth projective surface of general type.
We let $h^{p,s}=\dim H^{p,s}(S)$ be its $(p,s)$ Hodge number.  One has that $h^{1,0}$ is the irregularity $q$ of $S$ whilst $ h^{2,0}$ is its geometric genus  $p_g$.
We say that $S$ has  an irregular pencil of genus $b$ if  there is a morphism with connected fibres $\pi:S\to B$ over a curve $B$ of genus $b>0$.

In this paper we prove the following theorem:

\begin{theorem} \label{teorema}
If $S$ is a complex smooth projective surface of general type  without irregular pencils of genus $\geq 2$ and  $q=2^k+1$ then $h^{1,1}\geq 4q-3.$
\end{theorem}

We note that   $h^{1,1}\geq 3q-2$ for complex smooth projective surfaces of general type.  In fact 
from the Bogomolov-Myiaoka-Yau inequality $c_2\geq 3\chi$,  and the fact that $c_2=2-4q+2p_g+ h^{1,1}$  one obtains

\begin{equation}\label{initialineq} h^{1,1}\geq p_g+q+1.\end{equation}

Then, by \cite{Be},   $p_g\geq 2q-4$ and equality holds if and only if  $S$ is birational to a product of two curves one of which has genus $2$.  In this last case,  one has $c_2=4\chi$,  and $ p_g=2(q-2)$ hence $ h^{1,1}=4q-6$.   Since in this case also $q\geq 4$,  we  conclude that  always
 \begin{equation}\label{popa}h^{1,1}\geq 3q-2.
\end{equation}

For surfaces $S$ without  irregular pencils of genus $\geq 2$,  Lazarsfeld and Popa in \cite{lazpopa} reproved  inequality  (\ref{popa}) and also that, if $q$ is odd,
\begin{equation}\label{lazp} h^{1,1}\geq 3q-1.
\end{equation}

If there is an irregular pencil with connected fibres   $\pi:S\to B$ over a curve $B$ of genus $b$, the following inequality has been recently proven in \cite{jsfk}:

\begin{equation} 
h^{1,1}\geq 2b(q-b)+2+\sum (l(F)-1)
\end{equation}
where $l(F)$ is the number of the irreducible components  of a fiber $F$.

\bigskip

In some instances, for surfaces $S$ without  irregular pencils of genus $\geq 2$,    inequalities (\ref{popa})  and (\ref {lazp}) can be improved: 

\begin{itemize}
\item[-] in the case $q=3$, $h^{1,1}\geq 9$ by \cite{cau}. No examples of surfaces with $h^{1,1}=9$ are known; the symmetric product of a curve of genus $3$, however, has $h^{1,1}=10$;
\item[-] when $q=4$, in \cite{cau} it is proven that $h^{1,1}\geq 11$. The Schoen surface (see \cite{Schoen}  and also \cite{cmr}) realizes the smallest known value of $h^{1,1}=12$; the next known example, $h^{1,1}=17$, is given by the symmetric product of a curve of genus $4$;
 \item[-] for $q=5$ we have $h^{1,1}\geq 17$ again by \cite{cau}. No examples are known for $h^{1,1}<25$; the Fano surface of the lines of a smooth cubic $3-$fold has $h^{1,1}=25$;  the symmetric product of a curve of genus $5$ gives $h^{1,1}=26$.
 \end{itemize}

\medskip

Theorem \ref{teorema}  extends the above cases $q=3$ and $q=5$. 
To prove Theorem \ref{teorema}, we use the  the following linear algebra result to estimate the dimension of the (1,1) part of image of  the cup product map   $\bigwedge^2H^1(S,\mathbb C)\to H^2(S,\mathbb C)$:

 \begin{proposition} \label{spazio}  Let $\mathcal H$ be the real vector space of $q\times q$ complex hermitian matrices.
 
 If  $q=2^k+1,  k\geq 2$ and   $L\subset \mathcal H$ is a (real) vector subspace  such that
every  $X\in L\setminus\{0\}$ has at least 2 positive and 2 negative eigenvalues  then $\dim L\leq q^2-(4q-3).$ \label{mat} \end{proposition}

Theorem \ref{teorema} does not hold for the Schoen surface which has  $q\neq 2^k+1$. We notice the analogy with the role played by   $2-$adic valuations for the estimates of dimensions of spaces of  constant rank  matrices  in \cite{adams} and \cite{causin2}.

\bigskip
 
The paper is organized as follows.  In Section 2 we explain  how hermitian matrices relate to lower bounds for  $h^{1,1}$ and in Section 3 we prove Proposition \ref{spazio} and Theorem \ref{teorema}. Finally in Section 4 we give some corollaries of Theorem \ref{teorema}. 

\medskip

\section{Albanese variety and Hermitian matrices} \label{her}

  Given a complex variety  $V$ with a real structure, we denote by $V(\mathbb R)$ (or by $V_\mathbb R$, should the notation become too cumbersome) the locus of its real points.  

\medskip

Given a complex smooth projective  irregular surface of general type $S$,  let $A$ be the Albanese variety of $S$ and $a: S\la A$  the Albanese mapping.
Defining the pull-back map $$a^\ast : H^{1,1}(A)\la H^{1,1}(S),$$ denote by 
$K$ the kernel of $a^\ast$.  Then clearly  $h^{1,1}\geq q^2-\dim K$ and our strategy for proving Theorem \ref{teorema} is to give an upper bound for the dimension of $K$.

Complex conjugation  of forms gives a real structure to the vector spaces $H^{1,1}(A)$ and $H^{1,1}(S)$ and $a^\ast$ is a real map; denote by  $a^{\ast}_{\bR}$ the induced map:

$$a^{\ast}_{\bR} : H^{1,1}(A)_\bR \to H^{1,1}(S)_\bR.$$ 

 Since $K(\bR)=\ker (a^{\ast}_{\bR})$ we have  $\dim_\bC K= \dim_\bR K(\bR).$
As in \cite{cau}, we identify $H^{1,1}(A)$ with the vector space $M:=M(q,\bC)$  of $q\times q$ complex matrices 
and $H^{1,1}(A)_\bR$ with the real space $\mathcal  H\subset M$ of the hermitian matrices.
Remark that the real structure of $M$ is here given by the involution $B\to ^t\!\!\overline B,$ and $\mathcal H=M(\bR).$ 
In this way we identify $K(\bR)$ with a real subspace of $\mathcal H.$

\begin{definition} For any $X\in \mathcal H$ the minimal  inertia of $X$ is the integer $ m_X= \min \{n_+,n_-\}$
 where $(n_+,n_-)$ is the signature (or inertia) of $X.$
 \end{definition}
 
 Remark that $m_X=0$ if and only if $X$ is semidefinite, that for any non-zero real number 
 $\lambda$ we have $m_X=m_{\lambda X}$ and also that the rank of $X$ is $\geq 2m_X$.
 
The following statement (which obviously generalizes to higher dimensions, see \cite{cau}) is an essential ingredient for studying the dimension of $K$:

 \begin{proposition}\label{inerzia} If $S$ has no irrational pencils of genus $ \geq 2$ then $m_X>1$ for any $X\in K(\bR)\setminus\{0\}.$
 \end{proposition}

 \begin{proof} Suppose $\omega=i\sum_{s,k} x_{k,s}\ dz_s\wedge d{\bar z}_k\in K(\bR)\setminus\{0\}$, with $X=(x_{k,s})$ a hermitian matrix.  Let  as above $ (n_+,n_-)$ be the signature of $X$ and assume by contradiction that 
 $m_X= \min \{n_+,n_-\}\leq  1$; note that we may assume without restrictions (by taking $-\omega$, if necessary) that $n_-\leq 1.$  
 
By diagonalizing $X$ we may write
 $\omega= i( \sum_j\beta_j\wedge {\bar \beta}_j -\lambda \alpha \wedge{\bar\alpha}) $
 where $\alpha$ is a non zero $(1,0)$-form,  $\beta_j$ are independent $(1, 0)$-forms on $A,$ and  $\lambda\in \bR, $ $\lambda\geq 0$. Note that   $\lambda= 0$ gives the case $n_-=0.$ 
 
 The pullback via the Albanese map gives an identification of $H^{1,0}(A)$ with $H^{1,0} (S)$ and, by abuse of language, we set $a^\ast(\gamma)=\gamma$ for any $\gamma \in H^{1,0}(A)$. Since $a^\ast(\omega)=0,$ by integration on $S,$ we obtain:
$$0=\int_S a^\ast(\omega)\wedge{\alpha\wedge\overline \alpha}=i \sum_j\int _S\beta_j\wedge {\bar \beta}_j\wedge \alpha\wedge\bar\alpha.$$  
Since all the summands have the same sign it follows that
 $$0=i\int _S \beta_j\wedge {\bar \beta}_j\wedge \alpha\wedge{\bar\alpha}=
i \int_S\alpha\wedge\beta_j\wedge\overline{\alpha\wedge\beta_j}, \quad \hbox{ for any }j.$$ This gives by positivity $\alpha\wedge\beta_j=0.$
Since, by hypothesis, $S$ has no irrational pencils of genus $ \geq 2$,  the Castelnuovo - de Franchis theorem  (\cite{dF}, see also \cite {beauv}) gives $\alpha=\beta_j=0.$ Hence $\omega=0$, a contradiction.
 \end{proof}

 \section{Proof of Proposition \ref{spazio}  and Theorem \ref{teorema} }
 
Let  $M$  be the space  of $q\times q$ complex matrices  and
 $\mathcal  H\subset M$ the real space  of the hermitian matrices. Let $M_2\subset M$  be the locus of the matrices of rank $\leq 2$ and $\mathcal H_2=M_2(\bR)\subset \mathcal H$ be the set of hermitian matrices of rank
$\leq 2.$

Consider the projective space $\bP:=\bP(M)$ which is a complex projective space of dimension $N=q^2-1.$
 Let $\bP_\bR\subset \bP$ be the real projective space corresponding to $\mathcal H$
and $D_2\subset \bP$ the determinantal variety corresponding to $M_2:$
$$D_2=\{ \langle X\rangle \in \bP: X\in M_2\}.$$ Then we have
$$D_2(\bR)= D_2\cap \bP_\bR= \{ \langle X\rangle \in \bP: X\in \mathcal H_2\};$$ Moreover, 
$D_2(\bR)$ is the union of two components $D_1$ and $D_0$,  where $D_1$ is the closure of the locus of matrices with signature $(1,1)$, that is of $ \{\langle X\rangle \in D_2: m_X=1\}$, and $ D_0=\{\langle X\rangle \in D_2: m_X=0\}$. 
Note that the intersection $D_1\cap D_0$ corresponds to the hermitian matrices of rank $1.$

One has $\dim_\bC D_2=\dim_\bR D_2(\bR)= 4q-5$ and moreover, by  \cite{h-t}, and \cite{f-k}:
\begin {proposition} \label{dispari} The degree
of $D_2$ is odd if and only if $q=2^k+1$.
\end{proposition}

\begin{proof} 
By \cite{h-t}, the degree of $D_2$ is 
$$\deg D_2 = \prod_{j=0}^{q-3}\ \frac{ \displaystyle{{q+j} \choose {q-2}}}{ \displaystyle{{q-2+j} \choose {q-2}}}=  \prod_{j=0}^{q-3}\ \frac{(q+j-1)(q+j)}{(j+1)(j+2)}$$  and by \cite{f-k}, sect.~$6$,  this quantity is odd if and only if $q-2$ and $q-1$ have disjoint binary expansion, {\em i.e.} when $q=2^k+1.$
\end{proof}

 Let $h\in H^{1}(\bP_\bR,\bZ_2)$ be the class of a hyperplane. We recall that the ring of $\bZ_2$-cohomology of the real projective space $\bP_\bR$ is generated by $h$ with the relation $h^{N+1}=0$, that is 
$H^\ast (\bP_{\bR}, \bZ_2)\equiv \bZ_2[h]/h^{N+1}.$ It follows that the dual $\bZ_2$-cohomology class $[Y]$ of
any real algebraic variety $Y\subset \bP_{\bR}$ of dimension $n$ and odd degree is not trivial, hence equal to $h^{N-n}.$

In particular, for  $q=2^k+1$, by Proposition \ref{dispari}  the class $x=[D_2(\bR)]$  is
$$x= h^{N-(4q-5)}\in H^{N-(4q-5)}(\bP_\bR,\bZ_2) .$$

We assume from now on that $q=2^k+1$ and $q\geq 5$.

We now take into account the classes of  the components of $D_2(\bR)$:

 \begin{lemma} Let $y= [D_0]$ and $z=[D_1]$  in $ H^{N-(4q-5)}(\bP_\bR,\bZ_2)$.  Then $y=0$ and $z=x.$
 \end{lemma}
 \begin{proof} Recall that $x=y+z$. We want to    prove that $y=0.$ Consider   the hyperplane  $T$  in $ \bP_\bR$ whose points are the trace zero matrices,   
 $T:=\{ \langle X\rangle \in \bP_\bR :\  \tr  (X)=0\}.$  Then, since $D_0$ corresponds to semi-definite matrices, $T\cap D_0 =\emptyset.$ This implies that the cup-product $[T]\cdot[D_0]=h\cdot y=0$, hence $y=0 $ and consequently $z=x$.
   \end{proof}

 Now consider the identity matrix $I$ and the point $\bI=\langle I\rangle \in \bP_\bR\subset \bP.$
 We let $C_2$ be the cone over $D_2$ with vertex $\bI,$ $C_2=\{ \langle tX+sI\rangle : X\in M_2\hbox{ and } \langle t:s\rangle\in\bP^1(\bC)  \}$.

Then  $C_2$ is a variety of (complex) dimension $4q-4$ invariant under conjugation in 
$\bP$;  its real locus $C_2(\bR)$  is the union $C_1\cup C_0$  of two real cones with vertices at $\bI$ over, respectively, $D_1$ and $D_0$.
It is easy to verify  that  any  $\langle X\rangle \in C_1$ satisfies $m_X\leq 1.$

Furthermore, the degree of $C_2$ equals the degree of $D_2$ and therefore it is odd, for $q=2^k+1$.

\begin{lemma}\label{class} The cohomology class  $ [C_1] \in H^{N- (4q-4)}(\bP_\bR,\bZ_2)$
is non trivial, that is $[C_1]=h^{N- (4q-4)}$. In particular $C_1$ intersects any real projective subspace of codimension
$4q-4.$\end{lemma}
\begin{proof} Write $[C_2(\bR)]=[C_0]+[C_1].$ Since the degree of $C_2$ is odd one has
$[C_2(\bR)]\neq 0.$  We note that the cohomology class associated to a cone over a cycle is zero if and only if the cycle is zero and so   $[C_0]=0.$  \end{proof}

\begin{proof} [Proof of Proposition \ref{spazio}]  
Let  $L\subset \mathcal H$ be a  vector subspace such  that
 every  $X\in L\setminus\{0\}$ has minimal inertia  $m_X>1.$  Denote by $\bL\subset \bP_\bR$ its associate projective space and assume by contradiction that  $c=\codim L< 4q-3$.
It follows that $h^c\cdot [C_2]\neq 0$ and therefore
 $\bL\cap C_1\neq \emptyset$ by Lemma \ref{class}.  This is a contradiction since for every $\langle  X \rangle\in C_1$ we have $m_X\leq 1.$
 \end{proof}

 \begin{proof}  [Proof  of Theorem \ref{teorema}]   The statement follows immediately from  the discussion in Section 2, Propositions \ref{inerzia} and  Proposition \ref{spazio}.
  \end{proof}

\section{Applications}
Theorem \ref{teorema} has some applications. The first one concerns the still mysterious surfaces satisfying $p_g=2q-3$ (see \cite{mp}).

\begin{proposition}\label{2q-3} Let  $S$ be a minimal surface of general type without irregular pencils of genus $\geq 2$ and satisfying $p_g=2q-3$.  If $q=2^k+1,  k\geq 3$, then the linear system $|K_S|$ has a fixed part. \end{proposition}
\begin{proof}  One has $\chi(S)=q-2$ and so by  Theorem \ref{teorema},  $K_S^2<8\chi(S)$. The result then follows by  \cite [Theorem 1.2]{mp} (cf. also \cite{bnp}).  
 \end{proof}

\begin{remark}\rm{ The only known examples  of  surfaces of general type without irregular pencils of genus $\geq 2$ and satisfying $p_g=2q-3$  are the symmetric product of a genus 3 curve and the Schoen surface (\cite{Schoen}).  By \cite{mpp2}  no such surfaces exist for $q=5$  and by \cite {mpp} for $q\geq 6$ such surfaces will always have birational canonical map.  }
\end{remark}

We can obtain also a lower bound for  $h^{1,1}$ even when $q\neq 2^k+1$:

\begin{proposition}\label{h}  Let  $S$ be a  surface of general type without irregular pencils of genus $\geq 2$. If  $q=2^k+1+\epsilon$, with $0<\epsilon<2^k$, then  $h^{1,1}\geq 4q-3-4\epsilon$. \end{proposition}
\begin{proof}  Keep the notation of Section 2 and take a decomposition $\mathcal H=V\oplus W$, with $V=\{(x_{i,j})\in \mathcal H : x_{i,j}=0,\ \max(i,j) > 2^k+1\}$ and  $W$ any complementary subspace; since $V$ is naturally identified with the space of $(q-\epsilon)\times (q-\epsilon)$ hermitian matrices and $K(\bR)\subset \mathcal H$, then $\dim K(\bR)\leq \dim (K(\bR)\cap V) + \dim W \leq [(q-\epsilon)^2-(4(q-\epsilon) -3)]+[q^2-(q-\epsilon)^2] =q^2-(4q-3-4\epsilon).$   \end{proof}

\end{document}